\newcommand{\ignore}[1]{}

\ignore{
\documentclass[a4paper, UKenglish, cleveref, autoref, thm-restate]{lipics-v2021}
\usepackage[dvipsnames]{xcolor}
\usepackage{cite}
\RequirePackage{hyperref}
\hypersetup{colorlinks=true, urlcolor=Blue, citecolor=Green, linkcolor=BrickRed, breaklinks, unicode}
\usepackage{amsmath,amssymb,amsthm}
\pdfminorversion=7
\bibliographystyle{abbrv}
\author{Sagi Snir}{University of Haifa, Israel}{Department of Evolutionary Biology, University of Haifa, Haifa 3498838, Israel}{https://orcid.org/0000-0001-5833-7659}{}
\author{Raphael Yuster}{University of Haifa, Israel}{raphy@math.haifa.ac.il}{https://orcid.org/0000-0001-7550-6506}{}
\authorrunning{S. Snir and R. Yuster}
\Copyright{Sagi Snir and Raphael Yuster}
\category{} 
\nolinenumbers 
\hideLIPIcs  
}

\documentclass[11pt]{article}

\usepackage{amsmath,amssymb,amsthm,setspace,graphicx,array,tikz,url,blkarray,easyReview,bm}
\usepackage[text={6.5in,8.4in}, top=1.3in]{geometry}
\usepackage{hyperref,lineno} \nolinenumbers
\newtheorem{theorem}{Theorem}[section]
\newtheorem{proposition}[theorem]{Proposition}
\newtheorem{lemma}[theorem]{Lemma}

\newtheorem{conjecture}[theorem]{Conjecture}

\renewcommand{\L}{{\mathcal L}}
\newcommand{\T}{{\mathcal T}}
\newcommand{\Q}{{\mathcal Q}}
\newcommand{\e}{{\mathbf e}}  
\renewcommand{\b}{{\mathbf b}}

\author{Sagi Snir
	\thanks{Department of Evolutionary Biology, University of Haifa, Haifa 3498838, Israel (ssagi@research.haifa.ac.il).}
	\and
	Raphael Yuster
	\thanks{Department of Mathematics, University of Haifa, Haifa 3498838, Israel
		(raphy@math.haifa.ac.il).}
} 

\begin{document}

\title{Simultaneous separation in bounded degree trees 
}

\date{}

\maketitle 
	
\begin{abstract}
	It follows from a classical result of Jordan that every tree with maximum degree at most $r$ containing a vertex set labeled by $[n]$, has a single-edge cut which separates two subsets $A,B \subset [n]$ for which $\min\{|A|,|B|\} \ge (n-1)/r$. Motivated by the tree dissimilarity problem in phylogenetics, we consider the case of separating vertex sets of {\em several} trees: Given $k$ trees with maximum degree at most $r$, containing a common vertex set labeled by $[n]$, we ask for a single-edge cut in each tree which maximizes $min\{|A|,|B|\}$ where $A,B \subset [n]$ are separated by the corresponding cut at each tree. Denoting this maximum by $f(r,k,n)$ and
	considering the limit $f(r,k) = \lim_{n \rightarrow \infty} f(r,k,n)/n$ (which is shown to always exist)
	we determine that $f(r,2)=\frac{1}{2r}$ and determine that $f(3,3)=\frac{2}{27}$, which is already quite intricate. The case $r=3$ is especially interesting in phylogenetics and our result implies that any two (three) binary phylogenetic trees over $n$ taxa have a split at each tree which separates two taxa sets of order at least
	$n/6$ (resp. $2n/27$), and these bounds are asymptotically tight.

\ignore{  The advent of the genomic revolution provides us with means to
  reconstruct the evolutionary history of a set of species from more
  than a single data source. The latter may result in several trees
  over the same taxa set, each representing a different history, and
  the question of how different these trees are from one another is
  natural and imperative, encompassing
  non-trivial combinatorial structures, some of which are open problems for more than
  four decades, such as the classical conjecture of Bandelt and Dress
  concerning the minimum quartet distance.
  
  \emph{Splits} - edges in a tree that induce cuts on the vertex
  set - are a key concept in measuring tree similarity. The popular
  {\em Robinson-Foulds} measure, considers a common split only in the case
  it partitions the taxa set  identically in the two trees.
  To the best of our knowledge, the naturally and intriguing related variant focusing on subsets of the taxa set  has not been handled in this context. Specifically, are there
  (non-identical)  cuts in the trees
  separating (splitting) the taxa set with a {\em constant} fraction of
  shared taxa in each part. The latter question has a classical solution in the case of a single tree. Here we formally define and investigate this question for more than a single tree, a mathematically challenging problem. While the general formulation is in terms of arbitrary trees, the
  restriction to binary phylogenetic trees provides useful asymptotic tight
  bounds on the taxa-set size that any split of two or three binary
  phylogenetic trees share. In particular, any split of two binary phylogenetic trees separates two taxa sets of order at least $n/6$ , and any split of three binary phylogenetic trees separates two taxa sets of order at least $2n/27$, and these bounds are asymptotically tight.}
\vspace{3mm}

\end{abstract}


\section{Introduction}
\ignore{ Phylogenetic trees are used to describe the evolutionary history of a
set of species (also called {\em taxa}) where directed edges signify
parenthood and correspondingly directed  paths represent
ancestry. Under that setting of {\em rooted} trees, the
basic informational unit is a rooted triplet~\cite{ASSU81,SS2000}. When
time notion is problematic, as is usually the case, trees are represented as {\em unrooted} and the basic information is an {\em unrooted quartet}~\cite{Estabrook-systBiol-1985}. Henceforth, we restrict the
discussion only to unrooted trees. 

The flood of genomic data from many sources of information, enables reconstructing the evolutionary history of the same
set of organisms, but from each such information source
separately. Due to key evolutionary processes,
each such source may give rise to a different tree, so we obtain several trees over the
taxa-set under study. The latter calls for the meta-notion of \emph{tree
  dissimilarity} - quantitative parameters over the tree space measuring their
dissimilarity. We mention some of the mostly studied parameters. As a tree is uniquely defined by the set of its
constituting quartets, one such well-studied measure is the quartet
distance \cite{BD-1986}. Another measure is the maximum agreement subtree (MAST) that
represent the largest subset of the taxa set, under which two tree
induced by this subset, are the same~\cite{Kubicka-Kubicki-McMorris-1992,Finden-Gordon-1985-J-Classication}. As removal of every edge (a.k.a.\! split) in a tree induces a partition over the
leaf-set, and the set of these induced partitions uniquely defines the
tree, we may measure how many common (or different) partitions two
trees have. The latter, the Robinson-Foulds symmetric
difference~\cite{RF}, is a popular distance measure in practice but is
very sensitive to tree sizes~\cite{Avni-SystBiol-2015}.
As trees in general, and phylogenetic trees in particular, are
combinatorial objects, the study of these questions raises key
questions regarding the nature of these
measures~\cite{Huber-AML-2005,Grunewald-mathBioSci-2007,Bryant-Steel-AdAppMath-1995}.
Several of these parameters relate to classical fact (see below) that a binary phylogenetic tree has a split which separates the taxa set into relatively balanced disjoint subsets (each subset containing at least $1/3$ of the taxa).
Given several trees on the same taxa set, is it true that we can find splits in each tree which still guarantee a relatively balanced separation of a large taxa set? And if so, what are these guarantees? 
These are the main questions the present paper studies.
We now formalize our question more specifically, for general trees.}

A vertex or edge subset $S$ of a graph is a {\em separator} for disjoint vertex sets $A$ and $B$ if the removal of $S$ from the graph {\em separates} $A$ and $B$; i.e., there is no connected component intersecting both $A$ and $B$ following the removal of $S$.
Already Jordan \cite{jordan-1869} observed that every tree has a vertex when, once removed,
results in a graph with no connected component having more than half of the vertices.
Jordan's observation also yields the well-known result that every tree with maximum degree $r$
has a single edge that separates two vertex sets of order at least $\lceil (n-1)/r \rceil$ each.
Balanced small separators such as this (i.e, separating two vertex sets, each consisting of a constant fraction of the total number of vertices) are a basic building block of divide and conquer graph algorithms (see \cite{golumbic-1980}). The existence of balanced separators has been extended in highly nontrivial ways from trees to bounded tree-width graphs, as first shown independently by Halin \cite{halin-1976} and Robertson and Seymour \cite{RS-1986}, and
to fixed minor-free graphs as first shown for planar graphs by Lipton and Tarjan \cite{LT-1979}.
We note that all of these separator results can be extended to the vertex-weighted setting under appropriate assumptions on the weights.

Motivated by the notion of separators and by the tree-dissimilarity problem in phylogenetics which we mention later, in this paper we consider the case of edge-separating two subsets of vertices simultaneously in {\em several} trees.
Let us now be formal about the problem, and see that it immediately extends the aforementioned
classical result which is the case of a single tree.

Recall that an edge cut $e$ (sometimes called {\em split at $e$}) of a tree is the partition of the tree into the two disjoint trees obtained after removal of the edge $e$. The edge cut $e$ therefore induces
a partition $P_e=\{A,B\}$ of the vertices of $T$ (or any subset of vertices of $T$) into two parts, $A$ and $B$.
For two disjoint subsets of vertices $X,Y$, we say that $X$ and $Y$ are {\em separated} by $P_e$, if $X \subseteq A$ and $Y \subseteq B$, or vice versa.

Consider a family of trees $\T=\{T_1,\ldots,T_k\}$ each containing a common subset of vertices labeled by $[n]$. Notice that this setting already encompasses two important special cases:
(i) $[n]$ is the entire vertex set of each tree;
(ii) $[n]$ is the set of leaves of each tree.
Equivalently, we can view $\T$ as a multigraph whose vertex set contains a subset labeled by $[n]$ and whose edge set is the disjoint union of the edges of $k$ trees.

For $1 \le i \le k$, given an edge $e_i \in E(T_i)$, the cut of $T_i$ at $e_i$ partitions the common subset of vertices $[n]$ of $T_i$ into two parts which we denote by $P_{e_i}$.
We call $(e_1,\ldots,e_k)$ a {\em cut vector}. 
For disjoint subsets $X,Y$ of $[n]$, we say that $X,Y$ are {\em separated by the cut vector}
if they are separated by $P_{e_i}$ for each $1 \le i \le k$.
Notice that in the aforementioned graph-theoretic formulation, this means that any path in the multigraph connecting a vertex of $X$ and a vertex of $Y$ cannot be contained in a single $T_i$ after the cut edges are removed.

For integers $r \ge 2$ and $k \ge 1$, let $f(r,k,n)$ be the largest integer such that given any $k$ trees $\{T_1,\ldots,T_k\}$ each containing a common subset of vertices labeled by $[n]$, and where each tree has maximum degree at most $r$, there
is a cut vector, and there are disjoint subsets $X,Y$ of $[n]$, each of size at least $f(r,k,n)$, such that $X$ and $Y$ are separated by the cut vector; see Figure \ref{f:f32}.

\begin{figure}[!ht]
	\includegraphics[scale=0.7,trim=-50 340 173 110, clip]{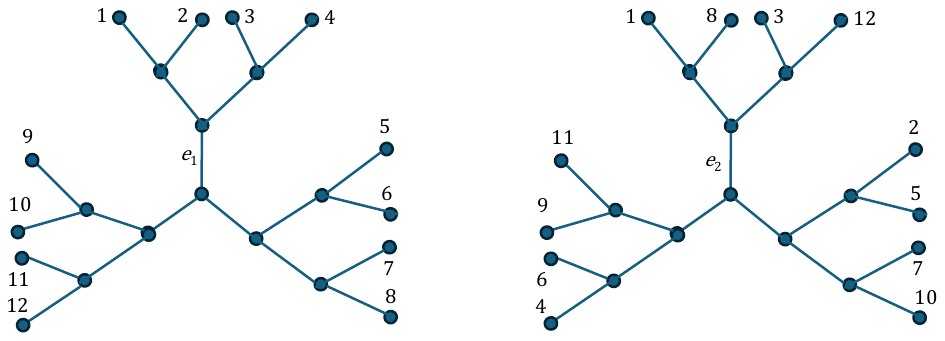}
	\caption{Two trees $T_1$ (left) and $T_2$ (right) with a common subset of vertices labeled by $[12]$ (in this case all vertices are common), both with maximum degree $3$. Any two disjoint subsets of vertices of size $3$ are not separated by a cut vector. This shows that $f(3,2,12) \le 2$. The depicted cut vector $(e_1,e_2)$ separates $X=\{1,3\}$ and $Y=\{5,6,7,9,10,11\}$.
	This shows that $f(\{T_1,T_2\},12) = 2$.}
	\label{f:f32}
\end{figure}  

More generally, for a family of trees $\T$ each containing a common subset of vertices labeled by $[n]$ we may ask for a cut vector of length $|\T|$, such that the smallest size of two disjoint subsets $X,Y$ of $[n]$ separated by the cut vector is maximized,
denoting this size by; again see Figure \ref{f:f32}. So, $f(r,k,n)$ is the minimum of $f(\T,n)$ taken over all such $k$-sized families of trees with bounded degree $r$. Naturally, we are especially interested in the case where $r,k$ are fixed while $n$ grows, so let
$$
f(r,k) = \lim_{n \rightarrow \infty} f(r,k,n)/n \,. ~\footnote{See Proposition \ref{prop:limit} for a proof that the limit exists.}
$$
It is straightforward to see that $f(r,1)=1/r$ (see Lemma \ref{l:simple}) and also quite easy to determine the case $r=2$, i.e., the case of paths,
for which we have $f(2,k)=1/2^k$ (see Proposition \ref{prop:2k}). However, determining $f(r,k)$ in general is considerably more involved. We solve this problem for the case $k=2$ and every degree $r$, and solve the next nontrivial case of three trees.
\begin{theorem}\label{t:main}
	$f(r,2)=\frac{1}{2r}$ and $f(3,3)=\frac{2}{27}$.
\end{theorem}
We strongly suspect (see Section \ref{sec:concluding}) that the case of three trees adheres to the following formula:
\begin{conjecture}\label{conj:1}
	$f(r,3)=(r-1)/r^3$ for all $r \ge 2$.
\end{conjecture}

As mentioned earlier, the case $r=3$ is of particular interest in phylogenetics. A short background follows.
Unrooted phylogenetic trees are used to describe the evolutionary relationship of a set of species (also called {\em taxa}), see \cite{SS-2003} for more information on phylogenetic trees. The flood of genomic data from many sources of information, enables reconstructing 
such trees, but each such information source may give rise to a different tree.
This calls for the notion of \emph{tree dissimilarity} - quantitative parameters over the tree space measuring their dissimilarity. We mention some of the mostly studied parameters in phylogenetics. As an unrooted phylogenetic tree is uniquely determined by the set of its constituting quartets (subsets on four taxa; see definition below), one such well-studied measure is the quartet distance \cite{BD-1986}. Another measure is the maximum agreement subtree (MAST) that represents the largest subset of the taxa set, under which two trees
induced by this subset, are the same~\cite{Kubicka-Kubicki-McMorris-1992,Finden-Gordon-1985-J-Classication}. As removal of every edge (a.k.a.\! split) in a tree induces a partition over the
leaf-set, and as the set of these induced partitions uniquely defines the
tree by the classical theorem of Buneman \cite{buneman-1971}, we may measure how many common (or different) partitions two
trees have. The latter, the Robinson-Foulds symmetric
difference~\cite{RF}, is a popular distance measure in practice but is
very sensitive to tree sizes~\cite{Avni-SystBiol-2015}.
Given several trees on the same taxa set, which, as mentioned, is a common theme in phylogenetic reconstruction, is it true that we can find splits in each tree which still guarantee a relatively balanced separation of a large taxa set, and if so, what are these guarantees? 

More formally, an unrooted binary phylogenetic tree \footnote{Hereafter, all unrooted phylogenetic trees are assumed binary.} is a tree whose internal vertices have degree $3$ and whose leaves
(taxa) are labeled with $[n]$; the internal vertices are unlabeled, but note that they determine the topology of the tree. Notice that an unrooted phylogenetic tree has precisely $n-2$ internal vertices and $2n-2$ vertices.  A {\em quartet} in a tree is a pair $X=\{a,b\}, Y=\{c,d\}$ where $a,b,c,d$ are taxa (leaves) such that $X$ and $Y$ are separated by some edge cut $e$ of the tree. We usually denote such a quartet by $ab|cd$. Quartet-based reconstruction is a major topic in phylogenetics, see \cite{ASY-2014} and the references therein. Given several phylogenetic trees on the same taxa set $[n]$ (recall that the leaves of each tree are now labeled with $[n]$ while the internal vertices are unlabeled), the goal is to find
the largest ``pieces'' of non-conflicting (i.e. compatible) information that is consistent with all the given trees. We may measure this by the number of common quartets, or, more generally, by the sizes
of common separated taxa sets. Indeed, as we shall see shortly, our aforementioned problem $f(3,k)$
precisely captures the asymptotic solution of the latter task.

Given a family $\T=\{T_1,\ldots,T_k\}$ of unrooted phylogenetic trees, each on taxa set $[n]$,
we may ask for a cut vector of length $|\T|$, such that the smallest size of two disjoint subsets $X,Y$ of taxa separated by the cut vector is maximized, denoting this size by $g(\T)$.
Let $g(k,n)$ be the minimum of $g(\T)$ ranging over all $k$-sized families of unrooted phylogenetic trees on taxa set $[n]$. We immediately obtain from the definition of $f(3,k,n)$ that $g(k,n) \ge f(3,k,n)$.
Given some adjustments of Theorem \ref{t:main}, it becomes not too difficult to prove that equality asymptotically holds and in fact:
\begin{theorem}\label{t:taxa}
	\begin{align*}
	\frac{n}{6} & \le g(2,n) \le (1+o(1))\frac{n}{6}\,,\\
	\frac{2n}{27} & \le g(3,n) \le (1+o(1))\frac{2n}{27}\,.
	\end{align*}
\end{theorem}
Notice that our result implies that for any two such phylogenetic trees, there are disjoint taxa sets $X,Y$ of size at least $n/6$ each, such that for any pair $a,b \in X$ and any pair $c,d \in Y$,
the quartet $ab|cd$ appears in both trees.
Notice that this is much larger than the more restricted requirement in the aforementioned MAST problem where there might only be a $\Theta(\log n)$-size taxa set $Z$ such that all quartets induced by $Z$ are the same in both trees, as proved by Markin \cite{markin-2020}.
 Similarly, for any three such phylogenetic trees, there are disjoint taxa sets $X,Y$ of size at least $2n/27$ each, such that for any pair $a,b \in X$ and any pair $c,d \in Y$, the quartet $ab|cd$ appears in all three trees.

The remainder of this paper consists of four additional sections. In Section 2 we prove our main result, Theorem \ref{t:main}, where we defer our most technical lemma, required for the lower bound of $f(3,3)$,
to Section 3. In Section 4 we consider unrooted phylogenetic trees and see how to adjust Theorem \ref{t:main} so as to obtain Theorem \ref{t:taxa}. Section 5 contains some concluding remarks and open problems.

\section{Proof of Theorem \ref{t:main}}

We begin with upper bound constructions for $f(r,2)$ and $f(r,3)$ which, in particular,
prove the upper bound part of Theorem \ref{t:main}. We first need to introduce some notation that will
be useful throughout the paper.

For a set of trees $T_1,\ldots,T_k$, each containing a subset of common vertices labeled by $[n]$,
consider a cut vector $ \e = (e_1,\ldots,e_k)$ and the corresponding partitions of $[n]$
denoted $P_{e_i}=\{X_{i,0},X_{i,1}\}$ for $1 \le i \le k$.
For a binary vector $\b \in \{0,1\}^k$, let
$X^{\e,\b}_0 = \cap_{i=1}^k X_{i,\b(i)}$ and $X^{\e,\b}_1=\cap_{i=1}^k X_{i,1-\b(i)}$.
For example, if $k=3$ and $\b=(0,1,0)$, then $X^{\e,\b}_0 = X_{1,0} \cap X_{2,1} \cap X_{3,0}$
and $X^{\e,\b}_1= X_{1,1} \cap X_{2,0} \cap X_{3,1}$.
Now, if $A,B$ are disjoint subsets of $[n]$ separated by $\e$,
then, by definition, there must be some vector $\b \in \{0,1\}^k$ such that $A \subseteq X^{\e,\b}_0$ and $B \subseteq X^{\e,\b}_1$. As we are looking at maximizing the sizes of $A$ and $B$, the problem
of finding $f(\{T_1,\ldots,T_k\},n)$ reduces to finding a cut vector $\e$ and a binary vector $\b \in \{0,1\}^k$ such that $\min \{ |X^{\e,\b}_0|\,,\, |X^{\e,\b}_1|\}$ is maximized.

\begin{lemma}\label{l:random}
	Let $k$ be fixed and let $T_1,\ldots,T_k$ be unlabeled trees, each on $\Theta(n)$ vertices, and
	sharing a subset $S$ of $n$ common vertices. Independently for each $1 \le i \le k$, randomly and bijectively assign the labels $[n]$ to $S$ in $T_i$, denoting the now-labeled set by $S_i$. Then whp\footnote{Throughout this paper whp (with high probability) is a shorthand for a sequence of probabilities that goes to $1$ as $n$ goes to infinity.} the following holds. For each cut vector $\e=(e_1,\ldots,e_k)$ and for each binary vector $\b \in \{0,1\}^k$,
	\begin{equation}\label{e:1}
		|X^{\e,\b}_0| = n \prod_{i=1}^k \frac{|X_{i,\b(i)}|}{n} + o(n)\;.
	\end{equation}
\end{lemma}
\begin{proof}
	We first note that the right hand side of \eqref{e:1}, without the $o(n)$ term, is the expected size of $X^{\e,\b}_0$ (see below). Thus the lemma states that whp, the actual size is very close to its expected value.
	As there are only $\Theta(n^k)$ possible cut vectors to consider, and $2^k$ choices for $\b$, it suffices to prove that
	for a given cut vector $\e$ and a given $\b$, the claimed equation \eqref{e:1} fails with probability at most
	$(Cn)^{-k}$. In fact, we will show that it fails with probability exponentially small in $n^{1/4}$.
	So fix $\e=(e_1,\ldots,e_k)$ and fix $\b \in \{0,1\}^k$. Notice that this fixes partitions
	of $S_i$ into $\{S_{i,0}, S_{i,1}\}$ for $1 \le i \le k$. After our random labeling of $S_i$,
	$\{S_{i,0}, S_{i,1}\}$ corresponds to a partition $(X_{i,0},X_{i,1})$ of $[n]$ but note that
	$|S_{i,0}|=|X_{i,0}|$ and $|S_{i,1}|=|X_{i,1}|$ are fixed and only depend on $\e$ and $\b$ and $S_i$. 
	So, for $t \in [n]$, we have $\Pr[t \in X_{i,j}] = |X_{i,j}|/n$ for $j=0,1$.
	
	Next, for $t \in [n]$, let $Y_t$ be the indicator random variable for the event $t \in X^{\e,\b}_0$. As the labels of $[n]$ are independently assigned to each $S_i$, we have
	$$
	\mathbb{E}[Y_t] = \Pr[t \in X^{\e,\b}_0] = \prod_{i=1}^k \frac{|X_{i,\b(i)}|}{n}\;.
	$$
	Since $|X^{\e,\b}_0| = \sum_{t \in [n]} Y_t$ we obtain
	$$
	{\mathbb E}[|X^{\e,\b}_0|] = n \prod_{i=1}^k \frac{|X_{i,\b(i)}|}{n}\;.
	$$
	But now consider the martingale corresponding to $|X^{\e,\b}_0|$
	where we expose the elements of $[n]$ (and their assignment to the trees) one by one. In other words, we expose the outcomes $Y_1,Y_2,Y_3,\ldots,Y_n$ in sequence. Before exposure, we have the expectation
	$\mathbb{E}[|X^{\e,\b}_0|] = n \prod_{i=1}^k \frac{|X_{i,\b(i)}|}{n}$ and after each exposure, we have the conditional expectation
	of $|X^{\e,\b}_0|$ given the outcome of the previously exposed $Y_t$'s
	(so the last element of the martingale sequence after all exposures is $|X^{\e,\b}_0|$).
	This martingale has length $n$ and satisfies the  Lipschitz condition
	(a single exposed outcome $Y_t$ cannot change the conditional expectation by more than $1$).
	Hence, by Azuma's inequality,
	$$
	\Pr[ |X^{\e,\b}_0| - {\mathbb E}[|X^{\e,\b}_0|] > n^{2/3}]  < e^{-\frac{n^{4/3}}{2n}} < e^{-n^{1/4}}\;.
	$$
	Consequently, $|X^{\e,\b}_0| = n\prod_{i=1}^k \frac{|X_{i,\b(i)}|}{n}+O(n^{2/3})$ fails with probability exponentially small in $n^{1/4}$.
\end{proof}
Notice that although Lemma \ref{l:random} and \eqref{e:1} is about $|X^{\e,\b}_0|$, it is also about $|X^{\e,\b}_1|$ since $X^{\e,\b}_1=X^{\e,\b^*}_0$ where $\b^*$ is $\b$'s ones' complement.

Suppose that $H$ is a fixed tree on $h$ vertices and $\ell$ leaves in which every vertex has degree at most $r$, and suppose that $n \ge \ell$. Let $\L(H)$ denote the set of leaves of $H$.
Replace each leaf $v$ with a path on $x_v \ge 1$ vertices
such that $\sum_{v \in \L(H)} x_v=n$. 
We call the new tree a {\em path blowup} of $H$. If all $x_v$ are equal (hence equal to $n/\ell$)
we call the new tree a {\em balanced path blowup} of $H$. Notice that a path blowup has $n+h-\ell$ vertices
and maximum degree at most $r$. We call the $h-\ell$ internal vertices of $H$ the {\em backbone} of the path blowup, as they remain intact following the blowup.

Having established the aforementioned notation and lemma, we are now ready to prove our upper bounds.
\begin{theorem}\label{t:upper}
	$f(r,2) \le \frac{1}{2r}$ and $f(r,3) \le \frac{(r-1)}{r^3}$.
\end{theorem}
\begin{proof}
	We begin with $f(r,2)$. Consider the following two constructions.
	Let $H_1$ be obtained from the star $K_{1,r}$ by ``splitting'' each leaf into two leaves;
	note that $H_1$ has $2r$ leaves and maximum degree $r$.
	Suppose that $2r$ divides $n$.
	Let $T_1$ be a balanced path blowup of $H_1$, so $T_1$ has $n+r+1$ vertices of which $r+1$ are backbone vertices. Recall that the balanced double star on $2(r-1)$ leaves is obtained by taking two vertex-disjoint stars $K_{1,r-1}$ and connecting their centers with an edge.
	Let $H_2$ be obtained by the balanced double star on $2(r-1)$ leaves by subdividing the edge connecting the two centers $r-1$ times; note that $H_2$ has $2(r-1)$ leaves and maximum degree $r$.
	Suppose that $2(r-1)$ divides $n$ (so, overall, $r(r-1)$ divides $n$).
	Let $T_2$ be the balanced path blowup of $H_2$, so $T_2$ has $n+r+1$ vertices of which $r+1$ are backbone vertices.
	
	Randomly and bijectively assign the labels $[n]$ to the non-backbone vertices of $T_1$
	and similarly (and independently) do that for the non-backbone vertices of $T_2$.
	So, the conditions of Lemma \ref{l:random} are met (each of the trees $T_1,T_2$ has $n+r+1=\Theta(n)$ vertices, of which $n$ are randomly assigned the label set $[n]$).
	
	We claim that whp $f(\{T_1,T_2\},n) \le n/2r +o(n)$.
	As each of $T_1, T_2$ has bounded degree $r$, showing this immediately yields $f(r,2) \le 1/2r$, as required.
	To see that whp $f(\{T_1,T_2\},n) \le n/2r +o(n)$ we shall use Lemma \ref{l:random},
	namely we shall assume that \eqref{e:1} holds.
	Consider any cut vector $\e=(e_1,e_2)$ and
	any binary vector $\b \in \{0,1\}^2$.
	Let us consider the various possible sizes of $X_{1,0},X_{1,1},X_{2,0},X_{2,1}$ and deduce 
	an upper bound for $\min \{ |X^{\e,\b}_0|\,,\, |X^{\e,\b}_1|\}$, whereby an upper bound for
	$f(\{T_1,T_2\},n)$.
	
	By our construction of $T_1$ we have that $|X_{1,0}| \in \{\frac{n}{r}, \frac{n(r-1)}{r},\alpha n\}$
	where $\alpha \le 1/2r$ or $\alpha \ge 1-1/2r$.
	By our construction of $T_2$ we have that $|X_{2,0}| \in \{\frac{n}{2}, \beta n\}$ where
	$\beta \le 1/2(r-1)$ or $\beta \ge 1-1/2(r-1)$. Recall also that $|X_{i,1}|=n-|X_{i,0}|$ for $i=1,2$.
	
	Let us now consider all possible combinations of the pair of numbers
	$|X_{1,\b(1)}| \cdot |X_{2,\b(2)}|$ and $|X_{1,1-\b(1)}|\cdot |X_{2,1-\b(2)}|$.
	By renaming and symmetry we have that the possible pairs of numbers, divided by $n^2$, are
	\begin{align*}
		(i) & ~ \{\frac{1}{r} \cdot \frac12\,,\, \frac{r-1}{r} \cdot \frac12\}\,, \\
		(ii) & ~ \{\frac{1}{r} \cdot \beta\,,\, \frac{r-1}{r} \cdot (1-\beta)\}\,, \\
		(iii) & ~ \{\alpha \cdot \frac12 \,,\, (1-\alpha) \cdot \frac12\}\,, \\
		(iv) & ~ \{\alpha \cdot \beta \,,\, (1-\alpha) \cdot (1-\beta)\}\,.
	\end{align*}
	Let $c$ denote the minimum of the two numbers in each of (i), (ii), (ii), (iv), respectively.
	We have for (i) that $c \le 1/2r$.
	For (ii) we have that if $\beta \le 1/2(r-1)$, then $c \le 1/2r(r-1)$ and if
	$\beta \ge 1-1/2(r-1)$, then $c \le ((r-1)/r)\cdot (1/2(r-1))=1/2r$. In any case,
	$c \le 1/2r$.
	For (iii) and (iv) we have (recalling that $\alpha \le 1/2r$ or $\alpha \ge 1-1/2r$) that
	$c \le \min\{\alpha,1-\alpha\} \le 1/2r$.
	Hence, in all cases (i) (ii), (iii), (iv), the minimum of the two numbers does not exceed $1/2r$.
	We therefore obtain that the product $\prod_{i=1}^2 \frac{|X_{i,\b(i)}|}{n}$ in \eqref{e:1} does not exceed $1/2r$,
	so by \eqref{e:1}, $|X^{\e,\b}_0| \le n/2r +o(n)$ (and analogously $|X^{\e,\b}_1| \le n/2r +o(n)$), as required.
	
	We proceed with the construction yielding the upper bound for $f(r,3)$.
	Let $T_1$ be the same tree constructed in the aforegiven proof of the upper bound for $f(r,2)$ and let $T_2$ and $T_3$ be additional copies of $T_1$.
	Randomly and bijectively assign the labels $[n]$ to the non-backbone vertices of $T_i$,
	independently for $i=1,2,3$, so, the conditions of Lemma \ref{l:random} are met.
	
	We claim that whp $f(\{T_1,T_2,T_3\},n) \le n(r-1)/r^3 +o(n)$.
	As each of $T_1,T_2,T_3$ has bounded degree $r$, showing this immediately yields $f(r,3) \le (r-1)/r^3$.
	To see that whp $f(\{T_1,T_2,T_3\},n) \le n(r-1)/r^3 +o(n)$ we shall use Lemma \ref{l:random},
	namely we shall assume that \eqref{e:1} holds.
	Consider any cut vector $\e=(e_1,e_2,e_3)$ and any binary vector $\b \in \{0,1\}^3$.
	Let us consider the various possible sizes of $X_{1,0},X_{1,1},X_{2,0},X_{2,1},X_{3,0},X_{3,1}$ and deduce an upper bound for $\min \{ |X^{\e,\b}_0|\,,\, |X^{\e,\b}_1|\}$, whereby an upper bound for
	$f(\{T_1,T_2, T_3\},n)$.
	
	By our construction of $T_i$ we have for each $i=1,2,3$ that $|X_{i,0}| \in \{\frac{n}{r}, \frac{n(r-1)}{r}, \alpha n\}$ where $\alpha \le 1/2r$ or $\alpha \ge 1-1/2r$.
	Recall also that $|X_{i,1}|=n-|X_{i,0}|$ for $i=1,2,3$.
	
	Let us now consider all possible combinations of the pair of numbers
	$|X_{1,\b(1)}|\cdot |X_{2,\b(2)}| \cdot  |X_{3,\b(3)}|$ and $|X_{1,1-\b(1)}|\cdot |X_{2,1-\b(2)}|\cdot |X_{3,1-\b(3)}|$.
	By renaming and symmetry we have that the possible pairs of numbers, divided by $n^3$ are of the form
	\begin{align*}
		(i) & ~ \{\frac{1}{r^3}\,,\, \frac{(r-1)^3}{r^3}\}\,, \\
		(ii) & ~ \{\frac{r-1}{r^3},\, \frac{(r-1)^2}{r^3}\}\,, \\
		(iii) & ~ \{\frac{\alpha}{r^2} \,,\, \frac{(1-\alpha)(r-1)^2}{r^2}\}\,, \\
		(iv) & ~ \{\frac{\alpha(r-1)}{r^2} \,,\, \frac{(1-\alpha)(r-1)}{r^2}\}\,, \\
		(v) & ~ \{\frac{\alpha(r-1)^2}{r^2} \,,\, \frac{1-\alpha}{r^2}\}\,, \\
		(vi) & ~ \{\frac{\alpha\beta}{r} \,,\, \frac{(1-\alpha)(1-\beta)(r-1)}{r}\}\,, \\
		(vii) & ~ \{\alpha\beta\gamma \,,\, (1-\alpha)(1-\beta)(1-\gamma)\}
	\end{align*}
	where $\alpha,\beta,\gamma$ are in $[0,1/2r] \cup [1-1/2r,1]$.
	It is easy to verify that the minimum in each of these possible pairs of numbers does not exceed
	$(r-1)/r^3$. We therefore obtain that the product $\prod_{i=1}^3 |X_{i,\b(i)}|/n$ in \eqref{e:1} does not exceed $(r-1)/r^3$,
	so by \eqref{e:1}, $|X^{\e,\b}_0| \le n(r-1)/r^3 +o(n)$ (and analogously $|X^{\e,\b}_1| \le n(r-1)/r^3 +o(n)$), as required.
\end{proof}

We now proceed with our lower bounds for $f(r,2)$ and $f(3,3)$. To this end, we first need to recall
the classical observation about trees with bounded degree $r$ whose proof is left as an exercise. 
\begin{lemma}\label{l:simple}
	Let $T$ be a tree with bounded degree $r$ and with a subset of vertices labeled by $[n]$. Then there is always a single edge cut $e$ such that the size of each part of the corresponding partition of $[n]$ is at least $(n-1)/r$. \qed
\end{lemma}
Note that the lemma implies that $f(r,1,n) \ge (n-1)/r$ and clearly $f(r,1,n) \le \lceil (n-1)/r \rceil$ by taking, say, a path blowup of $K_{1,r}$. Hence, we have $f(r,1)=1/r$, as mentioned in the introduction. We next prove a lower bound for $f(r,2,n)$.
\begin{proposition}\label{prop:lower}
	Suppose that $T_1$ and $T_2$ are two trees of bounded degree $r$, each having a common subset of vertices labeled by $[n]$.
	Then there is a cut vector $(e_1,e_2)$ and there are disjoint subsets $X,Y$ of $[n]$, each of size at least $(n-1)/2r$, such that $X$ and $Y$ are separated by the cut vector.
\end{proposition}
\begin{proof}
	For $i \in \{1,2\}$, consider some split $e_i$ of $T_i$ such that the
	corresponding partitions $P_{e_i}=\{X_{i,0}, X_{i,1}\}$ of $[n]$ satisfy Lemma \ref{l:simple},
	so that $|X_{i,j}| \ge (n-1)/r$.
	
	We say that the {\em type} of a vertex $t \in [n]$ is $(i,j)$ if $t \in X_{1,i}$ and $t \in X_{2,j}$.
	So there are four possible types: $(0,0), (0,1), (1,0), (1,1)$. Let $S_{i,j}$ denote the set of vertices
	of type $(i,j)$. Then we have $X_{1,0}=S_{0,0} \cup S_{0,1}$, $X_{1,1}=S_{1,0} \cup S_{1,1}$
	$X_{2,0} = S_{0,0} \cup S_{1,0}$ and $X_{2,1} = S_{0,1} \cup S_{1,1}$.
	We have that $S_{0,0}$ and $S_{1,1}$ are separated by $(e_1,e_2)$ and that $S_{0,1}$ and $S_{1,0}$
	are separated by $(e_1,e_2)$. It remains to prove that
	$$
	\max \{ \min \{|S_{0,0}|,|S_{1,1}|\} \, , \min \{|S_{0,1}|,|S_{1,0}|\} \} \ge \frac{n-1}{2r}\;.
	$$
	Indeed, suppose the left hand side is smaller than $(n-1)/2r$, then one of the following holds:
	$|S_{0,0}|+|S_{0,1}|$ or $|S_{0,0}|+|S_{1,0}|$ or $|S_{1,1}|+|S_{0,1}|$ or $|S_{1,1}|+|S_{1,0}|$
	is smaller than $(n-1)/r$. But these are exactly the sizes of the sets $X_{1,0}$, $X_{2,0}$, $X_{2,1}$, $X_{1,1}$ which, recall, are at least $(n-1)/r$ each, a contradiction.
\end{proof}

Notice that Proposition \ref{prop:lower} gives $f(r,2) \ge 1/2r$, so together with Theorem \ref{t:upper} we have that $f(r,2)=1/2r$ as required for Theorem \ref{t:main}.
As we have already proved $f(3,3) \le 2/27$ as a special case of Theorem \ref{t:upper},
it only remains to prove that $f(3,3) \ge 2/27$ to obtain Theorem \ref{t:main} in full.

\begin{theorem}\label{t:lower-3}
	Suppose that $T_1,T_2,T_3$ are trees of bounded degree $3$ each having a common subset of vertices labeled by $[n]$.
	Then there is a cut vector $(e_1,e_2,e_3)$ and disjoint subsets $X,Y$ of $[n]$, each of size at least $2n/27 - o(n)$, such that $X$ and $Y$ are separated by the cut vector.
	In particular, $f(3,3) \ge 2/27$.
\end{theorem}
To prove Theorem \ref{t:lower-3} we shall need the following technical lemma which is proved in the next section.
\begin{lemma}\label{l:lower-2}
	Suppose that $T_1,T_2$ are trees of bounded degree $3$ each having a common subset of vertices labeled by $[n]$.
	Then there is a cut vector $(e_1,e_2)$ and disjoint subsets $X,Y$ of $[n]$ separated by the cut vector,
	where each has size at least $4n/27-o(n)$ and the sum of their sizes is at least $4n/9-o(n)$.
\end{lemma}
\begin{proof}[Proof of Theorem \ref{t:lower-3}]
	Consider the sets $X$ and $Y$ guaranteed by Lemma \ref{l:lower-2} and the corresponding cut vector $(e_1,e_2)$ separating them.
	Then we have $z := |X| + |Y| \ge 4n/9-o(n)$ and $\min\{|X|, |Y|\} \ge 4n/27-o(n)$.
	By Lemma \ref{l:simple}, applied to the common set $X \cup Y$ labeled by $[z]$, there is an edge cut $e_3$ of $T_3$ and a corresponding partition $P_{e_3}=\{W,Z\}$ of $X \cup Y$ such that $|W| \ge z/3-o(n)$ and $|Z| \ge z/3-o(n)$. Hence, $\min\{|Z|, |W|\} \ge z/3-o(n) \ge 4n/27-o(n)$.
	Let $A = X \cap Z$, $B = X \cap W$, $C = Y \cap Z$, $D = Y \cap W$.
	Since $(e_1,e_2,e_3)$ separates $A$ and $D$ and also separates $B$ and $C$, it remains to prove that
	$$
	\max \{ \min \{|A|, |D|\} \, , \min \{|B|, |C|\} \} \ge \frac{2n}{27}-o(n)\;.
	$$
	Indeed, suppose the left hand side is smaller than $2n/27-o(n)$, then one of the following holds:
	$|A|+|B|$ or $|A|+|C|$ or $|D|+|B|$ or $|D|+|C|$
	is smaller than $4n/27-o(n)$. But these are exactly the sizes of the sets
	$X$, $Z$, $W$, $Y$ which, recall, are at least $4n/27-o(n)$ each, a contradiction.
\end{proof}

We have therefore completed the proof of Theorem \ref{t:main}. \qed	

\section{Proof of Lemma \ref{l:lower-2}}

For $S \subset [n]$, it will be convenient to define its {\em weight} as $w(S):=|S|/n$. Recall that the lemma asks to find disjoint subsets $X$ and $Y$ of $[n]$ and a cut vector $(e_1,e_2)$ separating them, such that
	\begin{eqnarray}
		w(X)+w(Y)~~ \ge & \frac{4}{9}-o_n(1)\,,\label{e:c1}\\
		\frac{4}{27}-o_n(1) ~~\le & w(X) \le w(Y)\label{e:c2}\,.
	\end{eqnarray}
	Notice that to show \eqref{e:c2} given \eqref{e:c1}, it suffices to show that $2w(Y) \ge 2w(X) \ge w(Y) -o_n(1)$.
		
	Consider a cut vector $(e_1,e_2)$ with partitions
	$P_{e_1}=\{X_{1,0}, X_{1,1}\}$ and $P_{e_2}=\{X_{2,0}, X_{2,1}\}$ of $[n]$ where
	$e_1$ maximizes $\min\{w(X_{1,0}),w(X_{1,1})\}$ and
	$e_2$ maximizes $\min\{w(X_{2,0}),w(X_{2,1})\}$.
	We may assume by Lemma \ref{l:simple} that $w(X_{i,j}) \ge 1/3-o_n(1)$ for $i=1,2$ and $j=0,1$.
	Since we may rename the trees and the parts, we may assume without loss of generality that
	$w(X_{1,0}) \le w(X_{i,j})$ for $i=1,2$ and $j=0,1$ and that $w(X_{2,0}) \ge w(X_{2,1})$.
	For notational simplicity, let $a=w(X_{1,0})$ and $b=w(X_{2,0})$ so we have that
	$w(X_{1,1})=1-a$, $w(X_{2,1})=1-b$, and
	$$
	\frac{1}{3}-o_n(1) \le a \le 1-b \le \frac{1}{2} \le b \le 1-a \le \frac{2}{3}+o_n(1)\,.
	$$
	Let $A=X_{1,0} \cap X_{2,0}$, $B=X_{1,0} \cap X_{2,1}$, $C=X_{1,1} \cap X_{2,0}$ and $D=X_{1,1} \cap X_{2,1}$.
	
	Hereafter until the end of this section, and in order to simplify notation, we shall assume for simplicity that {\em all} stated equalities and inequalities hold up to $o_n(1)$. So when we write, say, $(1-b)a \le 4/27$ we shorthand for $(1-b)a \le 4/27 \pm o_n(1)$.
	
	Our proof is divided into several cases related to the quantities $a,b$ and several additional quantities that we now define.
	
	Consider the edge $e_1$ which has one endpoint in $X_{1,0}$ and the other endpoint, denoted $y$, in
	$X_{1,1}$. So $y$ has at most two neighbors in $X_{1,1}$, denoting them $u$ and $v$ ($u$ may not exist
	if $y$ has only one neighbor in $X_{1,1}$). Consider the cut of $T_1$ at edge $yu$, denoting
	the part of the cut containing $u$ by $X_{1,1,0}$ and consider the cut of $T_1$ at edge $yv$, denoting the part of the cut containing $v$ by $X_{1,1,1}$ (if there is no $u$, then let
	$X_{1,1,0} = \emptyset$). We have $X_{1,1} = \{y\} \cup X_{1,1,0} \cup X_{1,1,1}$.
	Let $c = w(X_{1,1,0})$ and $d=w(X_{1,1,1})$ and we may assume by renaming that
	$c \le d$. Also observe that $a+c+d+w(y)=1$, but recall that in our
	notation ignoring $o_n(1)$ terms (and since $w(y)=1/n$ if $y \in [n]$ and $w(y)=0$ if $y$ is not in the common subset of vertices labeled by $[n]$), we shall write $a+c+d=1$. Also observe that $a \ge d$ by our assumption that $e_1$ maximizes $\min\{w(X_{1,0}),w(X_{1,1})\}$.
	Figure \ref{f:trees} illustrates our notations.
	
	\begin{figure}[!ht]
		\includegraphics[scale=0.6,trim=-120 300 395 120, clip]{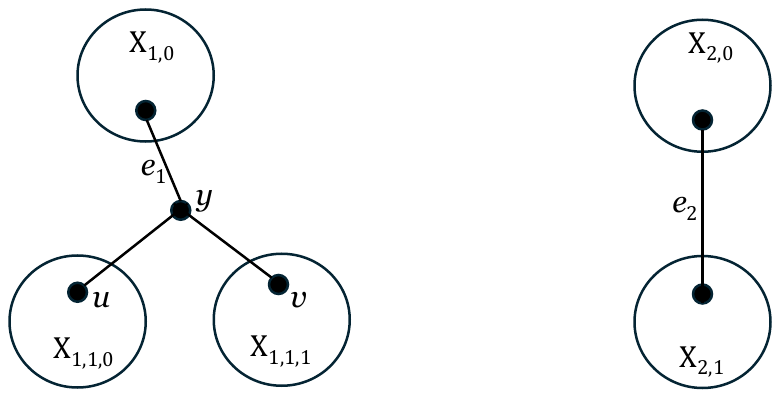}
		\caption{The trees $T_1$ (left) and $T_2$ (right), the cut vector $(e_1,e_2)$,
			and the partition sets; recall that $X_{1,1} = \{y\} \cup X_{1,1,0} \cup X_{1,1,1}$.}
		\label{f:trees}
	\end{figure}  
	
	We now list our cases (and subcases therein) and then prove \eqref{e:c1} and \eqref{e:c2} for each of them. It will be immediate from the sequel that these cover all possible cases,  and the necessity of this case partition. Notice that some of the constants defined in this list arise naturally from optimization, while others, such as the constant 0.7 in Case 2.2.2.2.1., are ad hoc choices (i.e., they are somewhat lenient, yet neccesary for range separation). 
	
	\begin{tabular}{@{}ll@{}}
	 1. &  $w(A) \ge ab$.\vspace{1mm}\\
	 2. &  $w(A) \le ab$.\vspace{1mm}\\
	 2.1. & $a(1-b) \ge \frac{4}{27}$.\vspace{1mm}\\
	 2.2. &  $a(1-b) \le \frac{4}{27}$ (so $a \le \frac{2}{\sqrt{27}}$ and $b \ge 1-4/27a$).\vspace{1mm}\\
	 2.2.1. & $w(X_{1,1,0} \cap X_{2,0}) := \beta \ge bc+\frac{1}{54}$.\vspace{1mm}\\
     2.2.1.1. & $\beta \ge (b+c)/2-\frac{5}{18}$.\vspace{1mm}\\
     2.2.1.2. & $\beta \le (b+c)/2-\frac{5}{18}$.\vspace{1mm}\\
     2.2.1.2.1. & $w(A) \ge (b-\beta)\frac{a}{a+d}$.\vspace{1mm}\\
     2.2.1.2.2. & $w(X_{1,1,1} \cap X_{2,0}) \ge (b-\beta)\frac{d}{a+d}$.\vspace{1mm}\\
     2.2.2. & $w(X_{1,1,1} \cap X_{2,0}) := \beta \ge bd-\frac{1}{54}$.\vspace{1mm}\\
     2.2.2.1. & $\beta \ge (b+d)/2-\frac{5}{18}$.\vspace{1mm}\\
     2.2.2.2. & $\beta \le (b+d)/2-\frac{5}{18}$.\vspace{1mm}\\
     2.2.2.2.1. & $w(A) \ge (b-\beta)\frac{a}{a+c} - z$ where $z=0$ if $c \ge 0.7a$ and $z=\frac{1}{108}$ if $c \le 0.7a$.\vspace{1mm}\\
     2.2.2.2.2. & $w(X_{1,1,1} \cap X_{2,0}) \ge (b-\beta)\frac{c}{a+c} + z$.
	\end{tabular}

\vspace{3mm}
	{\bf Case 1}:
	Here we assume $w(A) \ge ab$.
	Then $w(C) = b- w(A) \le b-ab=b(1-a)$, $w(D)=1-a-w(C) \ge 1-a-b(1-a) = (1-a)(1-b)$.
	Notice that $A$ and $D$ are separated by $(e_1,e_2)$.
	Then we have $w(A)+w(D) \ge ab+(1-a)(1-b)=1+2ab-a-b \ge 4/9$ since
	$a,b$ are both between $1/3$ and $2/3$, proving \eqref{e:c1}.
	Also notice that $ab \le (1-a)(1-b)$ since $a \le 1-b$. 
	So, to obtain \eqref{e:c2} it remains to prove that $ab \ge 4/27$ which clearly holds since $ab \ge 1/6$.
	
	{\bf Case 2}: Here we assume $w(A) \le ab$, and since $w(A)+w(B)=a$ we have $w(B) \ge a(1-b)$ and $w(C) \ge b(1-a)$. Notice that $B$ and $C$ are separated by $(e_1,e_2)$.
	Then we have $w(B)+w(C) \ge b+a-2ba \ge 4/9$ since
	$a,b$ are both between $1/3$ and $2/3$, proving \eqref{e:c1}.
	Also notice that $a(1-b) \le b(1-a)$ since $a \le 1-a$ and $b \ge 1-b$. 
	So, to obtain \eqref{e:c2} it remains to prove that $a(1-b) \ge 4/27$.
	However, note that this does not necessarily hold.
	
	{\bf Case 2.1}: $a(1-b) \ge 4/27$. Then by the previous line, we also have \eqref{e:c2}.
	
	{\bf Case 2.2}: This is the major subcase of the remainder of the proof, in which we assume
	$w(A) \le ab$ and $a(1-b) \le 4/27$. Notice that this implies that $a \le 2/\sqrt{27}=0.3849...$ and $b \ge 1-4/27a$, which we shall often use.

	Since $w(A) \le ab$ and since $b(a+c+d)=b$, we either have Case 2.2.1 where $w(X_{1,1,0} \cap X_{2,0}) \ge bc+1/54$ or
	else have Case 2.2.2 where $w(X_{1,1,1} \cap X_{2,0}) \ge bd-1/54$.
	
	{\bf Case 2.2.1}: Here we assume that $w(X_{1,1,0} \cap X_{2,0}):=\beta \ge bc+1/54$.
	The cut vector $(yu,e_2)$ separates $X_{1,1,0} \cap X_{2,0}$ and
	$(X_{1,0} \cup X_{1,1,1}) \cap X_{2,1}$.
	Also note that $w((X_{1,0} \cup X_{1,1,1}) \cap X_{2,1}) = a+d-(b-\beta)=1-c-b+\beta$.
	
	We next show that any of the weights of these two separated parts, namely $\beta$
	and $1-c-b+\beta$ are within a factor of $2$ of each other. Recall that this suffices in order to prove prove \eqref{e:c2} assuming \eqref{e:c1} holds.
	So, we first show that $2\beta \ge 1-c-b+\beta$ or, equivalently, $\beta \ge 1-c-b$.
	Since $\beta \ge bc+1/54$, it suffices to show that $bc+1/54 \ge 1-c-b$,
	equivalently $b \ge (1-c-1/54)/(1+c)$.
	Now, since $c \le d \le a$ and $a+c+d=1$, we have that $c \ge 1-2a$ so it suffices to prove that
	$ b \ge (a-1/108)/(1-a)$. But recall that 
	$b \ge 1-4/27a$ so it suffices to prove that
	$1-4/27a \ge (a-1/108)/(1-a)$. This translates to showing that $54a^2-31.25a+4 \le 0$ which holds
	for all $0.2 \le a \le 0.3865$. But recall that we have $1/3 \le a \le 2/\sqrt{27} = 0.3849...$, so the claim follows.
	We also need to show that $\beta \le 2(1-c-b+\beta)$.
	Equivalently $\beta \ge 2(b+c-1)$. It suffices to prove that $bc + 1/54 \ge 2b+2c-2$. But $c \le a \le 1-b$ so $2b+2c-2 \le 2b+2(1-b)-2=0$.
	
	We have shown that the weights of these two separated parts, $\beta$
	and $1-c-b+\beta$, are within a factor of $2$ of each other, but we still need to show that 
	the sum of their weights is at least $4/9$ to also obtain \eqref{e:c1}.
	So, we must show that $1-c-b+2\beta \ge 4/9$. However, this only holds if $\beta \ge 2/9+b/2-1/2+c/2 = (b+c)/2 - 5/18$.
	
	{\bf Case 2.2.1.1}: $\beta \ge (b+c)/2 - 5/18$, so as mentioned in the previous line,
	we also have \eqref{e:c1}.
	
	{\bf Case 2.2.1.2}: Here we assume that $\beta \le (b+c)/2 -5/18$.
	Since $w(A)+w(X_{1,1,1} \cap X_{2,0})=b-\beta$, we either have Case 2.2.1.2.1 that $w(A) \ge (b-\beta)a/(a+d)$ or else have Case 2.2.1.2.2 that $w(X_{1,1,1} \cap X_{2,0}) \ge  (b-\beta)d/(a+d)$.
	
	{\bf Case 2.2.1.2.1, proving \eqref{e:c1}}: Here we assume that $w(A) \ge (b-\beta)a/(a+d)$.
	Then $w(D) \ge c+d-(b-(b-\beta)a/(a+d)) = c+d-b+(b-\beta)a/(a+d)$.
	Since $A$ and $D$ are separated by $(e_1,e_2)$, we have two separated parts with respective sizes
	at least $(b-\beta)a/(a+d)$ and at least $c+d-b+(b-\beta)a/(a+d)$.
	Then we have $w(A)+w(D) \ge c+d-b+2ba/(a+d)-2\beta a/(a+d)$.
	Using $\beta \le (b+c)/2 -5/18$, we therefore have that
	$w(A)+w(D) \ge c+d-b+2ba/(a+d)-(2a/(a+d))((b+c)/2-5/18)$.
	Substituting $c=1-a-d$ we have that $w(A)+w(D) \ge (5a-9bd+9d)/9(a+d)$ and we have to prove it is at least $4/9$.
	This translates to proving $5a-9bd+9d \ge 4a+4d$ equivalently $a+5d \ge 9bd$,
	equivalently $a \ge d(9b-5)$. Suffices to prove that $a \ge a(9b-5)$,
	equivalently $b \le 2/3$, which always holds.
	
	{\bf Case 2.2.1.2.2, proving \eqref{e:c1}}: Here we assume that $w(X_{1,1,1} \cap X_{2,0}) \ge  (b-\beta)d/(a+d)$.
	Then $w((X_{1,0} \cup X_{1,1,0}) \cap X_{2,1}) \ge c+a-b+(b-\beta)d/(a+d))$.
	Since $X_{1,1,1} \cap X_{2,0}$ and $(X_{1,0} \cup X_{1,1,0}) \cap X_{2,1}$ are separated
	by $(yv,e_2)$, we have two separated parts with respective sizes
	at least $(b-\beta)d/(a+d)$ and at least $c+a-b+(b-\beta)d/(a+d)$.
	Then we have $w(X_{1,1,1} \cap X_{2,0}) + w((X_{1,0} \cup X_{1,1,0}) \cap X_{2,1}) \ge c+a-b+2bd/(a+d)-2\beta d/(a+d)$.
	Using $\beta \le (b+c)/2 -5/18$, we therefore have that this sum is at least 
	$c+a-b+2bd/(a+d)-(2d/(a+d))((b+c)/2-5/18)$.
	Substituting $c=1-a-d$ this expression is equivalent to
	$(5d-9ba+9a)/9(a+d)$ and we have to prove it is at least $4/9$.
	This translates to proving $d \ge a(9b-5)$ or
	$b \le (d/a+5)/9$. But we know that $b \le 1-a$ and $1-a=d+c \le 2d$ so $d \ge (1-a)/2$, hence it suffices to prove
	that $1-a \le ((1-a)/2a+5)/9$ which indeed holds for all $a \ge 1/3$.
	
	{\bf Cases 2.2.1.2.1 and 2.2.1.2.2, proving \eqref{e:c2}}: 
	Notice that in 2.2.1.2.1, the size
	$(b-\beta)a/(a+d)$ is not larger than the size $c+d-b+(b-\beta)a/(a+d)$
	since $c+d-b=1-a-b \ge b-b \ge 0$ and in 2.2.1.2.2, the size $(b-\beta)d/(a+d)$ is not larger
	than the size $c+a-b+(b-\beta)d/(a+d))$ since $c+a-b \ge c+d-b =1-a-b \ge b-b \ge 0$.
	Furthermore, since $(b-\beta)d/(a+d)$ is not larger than $(b-\beta)a/(a+d)$, to prove
	\eqref{e:c2} in both cases it only remains to show that $(b-\beta)d/(a+d) \ge 4/27$.
	As $\beta \le (b+c)/2 -5/18$, it suffices to prove that
	$(b+5/18-(b+c)/2)d/(a+d) \ge 4/27$, equivalently $(b/2-c/2+5/18)d/(a+d) \ge 4/27$,
	equivalently $d(9a+9b+9d-4) \ge (8/3)(a+d)$. 
	Recalling that $b \ge 1-4/27a$, this is equivalent to proving that the surface
	$F(a,d)=a^2(8 - 27 d) + 4 d - a d (7 + 27 d)$ is not positive
	for $1/3 \le  a \le 2/\sqrt{27}$ and $(1-a)/2 \le d \le a$.
	It is straightforward to verify that $\nabla F \neq (0,0)$ in this region,
	hence the maximum of $F(a,d)$ is attained on the boundary.
	When $d=a$ we have $F(a,a)=-54 a^3 + a^2 + 4 a$ which is negative for
	$1/3 \le  a \le 2/\sqrt{27}$ and when $d=(1-a)/2$ we have
	$F(a,(1-a)/2)=(8 - 49 a + 46 a^2 + 27 a^3)/4$ which is also negative for
	$1/3 \le  a \le 2/\sqrt{27}$. When $a=1/3$ we must have $d=1/3$
	and $F(1/3,1/3)$ is negative. When $a=2/\sqrt{27}$ we have
	$F(d,2/\sqrt{27})=-\frac{2}{27}(-16 + 21 \sqrt{3}d + 81\sqrt{3}d^2)$
	which is always negative for $d \ge 1/2-1/\sqrt{27}$.
	
	{\bf Case 2.2.2}:  Here we assume that $w(X_{1,1,1} \cap X_{2,0})=\beta \ge bd-1/54$.
	The cut vector $(yv,e_2)$ separates $X_{1,1,1} \cap X_{2,0}$ and
	$(X_{1,0} \cup X_{1,1,0} ) \cap X_{2,1}$.
	Also note that $w((X_{1,0} \cup X_{1,1,0}) \cap X_{2,1}) = a+c-(b-\beta)=1-d-b+\beta$.

	We next show that any of the weights of these two separated parts, namely $\beta$
	and $1-d-b+\beta$ are within a factor of $2$ of each other. Notice that this suffices in order to prove prove \eqref{e:c2} assuming \eqref{e:c1} holds.
	So, we first show that $2\beta \ge 1-d-b+\beta$ or, equivalently, $\beta \ge 1-d-b$.
	Since $\beta \ge bd-1/54$, it suffices to show that $bd-1/54 \ge 1-d-b$,
	equivalently $b \ge (1-d+1/54)/(1+d)$.
	Now, since $d \ge (1-a)/2$, it suffices to prove that
	$b \ge (1+a+1/27)/(3-a)$. But recall that $b \ge 1-4/27a$ so it suffices to prove that
	$1-4/27a \ge (1+a+1/27)/(3-a)$. This translates to showing that $54a^2-57a+12 \le 0$ which
	indeed holds for all $0.291 \le a \le 0.765$ so in particular in our range of interest.
	We also need to show that $\beta \le 2(1-d-b+\beta)$.
	Equivalently $\beta \ge 2(b+d-1)$. It suffices to prove that $bd - 1/54 \ge 2b+2d-2$.
	But $d \le a \le 1-b$ so $2b+2d-2 \le 2b+2(1-b)-2=0$ while $b \ge 1/2$ and $d \ge (1-a)/2 \ge 1/6$
	so $bd \ge 1/12 \ge 1/54$.
	
	We have shown that the weights of these two separated parts, $\beta$
	and $1-d-b+\beta$, are within a factor of $2$ of each other, but we still need to show that 
	the sum of their weights is at least $4/9$ to also obtain \eqref{e:c1}.
	So, we must show that $1-d-b+2\beta \ge 4/9$. However, this only holds if
	$\beta \ge 2/9+b/2-1/2+d/2 = (b+d)/2 - 5/18$.
	
	{\bf Case 2.2.2.1}: $\beta \ge (b+d)/2 - 5/18$, so as mentioned in the previous line,
	we also have \eqref{e:c1}.
	
	{\bf Case 2.2.2.2}: Here we assume that $\beta \le (b+d)/2 -5/18$.	
	Since $w(A)+w(X_{1,1,0} \cap X_{2,0})=b-\beta$, we either have Case 2.2.2.2.1 that $w(A) \ge (b-\beta)a/(a+c)-z$ or else have Case 2.2.2.2.2 that $w(X_{1,1,0} \cap X_{2,0}) \ge  (b-\beta)c/(a+c)+z$ where we define $z$ to be $0$ if $c \ge 0.7a$ and define $z$ to be $1/108$ if $c \le 0.7a$.
	
	{\bf Case 2.2.2.2.1, proving \eqref{e:c1}}: Here we assume that $w(A) \ge (b-\beta)a/(a+c)-z$.
	Then $w(D) \ge c+d-(b-(b-\beta)a/(a+c)+z) = c+d-b+(b-\beta)a/(a+c)-z$.
	Since $A$ and $D$ are separated by $(e_1,e_2)$, we have two separated parts with respective sizes
	at least $(b-\beta)a/(a+c)-z$ and at least $c+d-b+(b-\beta)a/(a+c)-z$.
	Then we have $w(A)+w(D) \ge c+d-b+2ba/(a+c)-2\beta a/(a+c)-2z$.
	Using $\beta \le (b+d)/2 -5/18$, we therefore have that
	$w(A)+w(D) \ge c+d-b+2ba/(a+c)-(2a/(a+c))((b+d)/2-5/18)-2z$.
	Substituting $d=1-a-c$ we have that $w(A)+w(D) \ge (5a-9bc+9c)/9(a+c)-2z$ and we have to prove it is at least $4/9$.
	Consider first the case $c \ge 0.7a$ where we have $z=0$.
	In this case, this translates to proving $5a-9bc+9c \ge 4a+4c$ equivalently $a+5c \ge 9bc$,
	equivalently $a \ge c(9b-5)$. Suffices to prove that $a \ge a(9b-5)$,
	equivalently $b \le 2/3$, which always holds.
	Consider next the case $c \le 0.7a$ where we have $z=1/108$.
	In this case, this translates to proving $5a-54bc+29c \ge 0$ equivalently
	$5a \ge c(54b-29)$, so suffices to prove that $5 \ge 0.7(54b-29)$ which indeed holds for all
	$b \le 0.669$ and recall that in our case $b \le 2/3$.
	
	{\bf Case 2.2.2.2.2, proving \eqref{e:c1}}: 
	Here we assume that $w(X_{1,1,0} \cap X_{2,0}) \ge  (b-\beta)c/(a+c)+z$.
	Then $w((X_{1,0} \cup X_{1,1,1}) \cap X_{2,1}) \ge d+a-b+(b-\beta)c/(a+c))+z$.
	Since $X_{1,1,0} \cap X_{2,0}$ and $(X_{1,0} \cup X_{1,1,1}) \cap X_{2,1}$ are separated
	by $(yu,e_2)$ we have two separated parts with respective sizes
	at least $(b-\beta)c/(a+c)+z$ and at least $d+a-b+(b-\beta)c/(a+c))+z$.
	Then we have $w(X_{1,1,0} \cap X_{2,0}) + w(w((X_{1,0} \cup X_{1,1,1}) \cap X_{2,1})) \ge d+a-b+2bc/(a+c)-2\beta c/(a+c)+2z$.
	Using $\beta \le (b+d)/2 -5/18$, we therefore have that this sum is at least $d+a-b+2bc/(a+c)-(2c/(a+c))((b+d)/2-5/18)+2z$.
	Substituting $d=1-a-c$ this expression is equivalent to
	$(5c-9ba+9a)/9(a+c)+2z$ and we have to prove it is at least $4/9$.
	As $z \ge 0$, it suffices to prove that $(5c-9ba+9a)/9(a+c) \ge 4/9$
	which translates to proving $c \ge a(9b-5)$ or
	$b \le (c/a+5)/9$. But we know that $b \le 1-a$ and $c \ge 1-2a$ so it suffices to prove
	that $1-a \le ((1-2a)/a+5)/9$ which translates to $9a^2-6a+1 \ge 0$ which always holds.
		
	{\bf Cases 2.2.2.2.1 and 2.2.2.2.2, proving \eqref{e:c2}}: 
	Notice that in 2.2.2.2.1, the size
	$(b-\beta)a/(a+c)-z$ is not larger than the size $c+d-b+(b-\beta)a/(a+c)-z$
	since $c+d-b=1-a-b \ge b-b \ge 0$ and in 2.2.2.2.2, the size $(b-\beta)c/(a+c)+z$ is not larger
	than the size $d+a-b+(b-\beta)c/(a+c))+z$ since $d+a-b \ge c+d-b =1-a-b \ge b-b \ge 0$.
	It remains to prove that $(b-\beta)a/(a+c)-z \ge 4/27$ and that
	$(b-\beta)c/(a+c)+z \ge 4/27$.
	
	Consider first the case $c \ge 0.7a$ in which case $z=0$.
	Since $a \ge c$, it suffices to prove that $(b-\beta)c/(a+c) \ge 4/27$.
	As $\beta \le (b+d)/2 -5/18$, it suffices to prove that
	$(b+5/18-(b+d)/2)c/(a+c) \ge 4/27$, equivalently $(b/2-d/2+5/18)c/(a+c) \ge 4/27$,
	Using $d=1-a-c$ this is equivalent to proving $c(9a+9b+9c-4) \ge (8/3)(a+c)$. 
	Recalling that $b \ge 1-4/27a$, this is equivalent to proving that the surface
	$F(a,c)=a^2(8 - 27 c) + 4 c - a c (7 + 27 c)$ is not positive
	for $1/3 \le  a \le 2/\sqrt{27}$ and $\max\{1-2a,0.7a\} \le c \le (1-a)/2$.
	Since $\max\{1-2a,0.7a\}$ is minimized when $a=10/27$ we have two possible regions to check:
	The region $10/27 \le a  \le 2/\sqrt{27}$ and $0.7a \le c \le (1-a)/2$ or the region
	$1/3 \le a \le 10/27$ and $1-2a \le c \le (1-a)/2$.
	It is straightforward to verify that $\nabla F \neq (0,0)$ in both regions,
	hence the maximum of $F(a,c)$ is attained on the boundary.
	When $c=0.7a$ we have $F(a,0.7a)=a(2.8 + 3.1 a - 32.13 a^2)$ which is negative for
	all $a \ge 10/27$.
	When $c=(1-a)/2$ we have $F(a,(1-a)/2)=(1/4)(8 - 49 a + 46 a^2 + 27 a^3)$
	which is negative for $0.21 \le a \le 0.56$ which contains the interval $10/27 \le a  \le 2/\sqrt{27}$.
	When $c=1-2a$ we have $F(a,1-2a)=4 - 42 a + 103 a^2 - 54 a^3$
	which is negative for $0.14 \le a \le 0.383$ which contains the interval
	$1/3 \le a \le 10/27$.
	When $a=2/\sqrt{27}$ we have $1.4/\sqrt{27} \le c \le 1/2-1/\sqrt{27}$ 
	and $F(2/\sqrt{27},c)=-(2/27)(-16 + 21 \sqrt{3} c + 81 \sqrt{3} c^2)$
	which is negative for all $c \ge 0.23$ and in particular for $c \ge 1.4/\sqrt{27}$.
	When $a=10/27$ we have $7/27 \le c \le 17/54$
	and $F(10/27,c)=800/729 - (62 c)/27 - 10 c^2$ which is negative for all $c \ge 0.236$ 
	and in particular for $c \ge 7/27$.	
	
	Consider next the case $c \le 0.7a$ in which case $z=1/108$,
	where we must prove that $(b-\beta)a/(a+c)\ge 4/27+1/108$ and that
	$(b-\beta)c/(a+c)\ge 4/27-1/108$. Since $c/a \le 0.7 \le (4/27-1/108)/(4/27+1/108)$,
	it suffices to prove that $(b-\beta)c/(a+c)\ge 4/27-1/108=15/108$.
	As $\beta \le (b+d)/2 -5/18$, it suffices to prove that
	$(b+5/18-(b+d)/2)c/(a+c) \ge 15/108$, equivalently $(b/2-d/2+5/18)c/(a+c) \ge 15/108$.
	Using $d=1-a-c$ and $b \ge 1-4/27a$, , this is equivalent to proving that the surface
	$F(a,c)=8c-3a^2(18c-5)-3ac(5+18c)$ is not positive for
	for $1/3 \le  a \le 2/\sqrt{27}$ and $1-2a \le c \le \min\{0.7a,(1-a)/2\}$.
	Notice that $(1-a)/2 \ge 0.7a$ for $a \le 5/12$ so in particular for $a \le 2/\sqrt{27}$
	so suffices to consider $1-2a \le c \le 0.7a$. Also notice that $1-2a \le 0.7a$ whenever
	$a \ge 10/27$ so suffices to assume $10/27 \le  a \le 2/\sqrt{27}$.
	It is straightforward to verify that $\nabla F \neq (0,0)$ in this region,
	hence the maximum of $F(a,c)$ is attained on the boundary.
	when $c=0.7a$ we have $F(a,0.7a)=a(-64.26 a^2 + 4.5 a + 5.6)$ which is negative for all
	$a \ge 0.3323$. For $c=1-2a$ we have $F(a,1-2a)=-108 a^3 + 207 a^2 - 85 a + 8$
	which is negative for $0.136 \le a \le 0.392$.
	For $a=10/27$ we must have $c=7/27$ and $F(10/27,7/27) < 0 $.
	For $a=2/\sqrt{27}$ we have $0.230... = 1-4/\sqrt{27} \le c \le 1.4/\sqrt{27}$
	and $F(2/\sqrt{27},c) = -12 \sqrt{3}c^2 - (10c)/\sqrt{3}+ 20/9$ which is negative for all
	$c \ge 0.21$.
	
	We have now completed the proof of \eqref{e:c1} and \eqref{e:c2} in all cases, so Lemma \ref{l:lower-2} follows. \qed
	
\section{Unrooted phylogenetic trees}

In this section we prove Theorem \ref{t:taxa} which, as we have already mentioned, follows from
Theorem \ref{t:main} by some relatively minor modifications. In fact, we need modifications
in both the upper and lower bounds.

\begin{proof}[Proof of Theorem \ref{t:taxa}]

We start with the lower bound where we need to show that $g(2,n) \ge n/6$ and $g(3,n) \ge 2n/27$.
Recall from the introduction that we already have $g(k,n) \ge nf(3,k)(1+o_n(1))$, in particular,
$g(2,n) \ge (1+o_n(1))n/6$ and $g(3,n) \ge (1+o_n(1))2n/27$, so we only need to eliminate
the $o_n(1)$ error term. It is easy to see that this error term stems from three locations:
Lemma \ref{l:simple}, the proof of Theorem \ref{t:lower-3} (which uses Lemma \ref{l:lower-2})
and the weight of vertex $y$ in the proof of Lemma \ref{l:lower-2} (see again Figure \ref{f:trees}) which may not be zero in that setting. So let us see that in the setting of unrooted phylogenetic trees, there are corresponding variants with no error term.

Starting with the variant for Lemma \ref{l:simple}, its error-term free analogue here is the following:
Let $T$ be an unrooted phylogenetic tree on $n$ leaves (taxa) labeled with $[n]$.
Then we must show that there is always an edge cut $e$ such that the number of leaves in each part of the partition $P_e$ is at least $n/3$. Indeed, this is the well known $(1/3,2/3)$ property of
phylogenetic trees whose proof can be seen, e.g., in Lemma 3.1 of \cite{SY-2012}.
Now, given this error-term free analogue of Lemma \ref{l:simple}, we immediately
obtain the error-term free analogue of Proposition \ref{prop:lower}, which is the following:
Suppose that $T_1,T_2$ are two unrooted phylogenetic trees on common taxa set $[n]$. Then there is a cut vector
$(e_1,e_2)$ and disjoint taxa sets $X,Y$ each with at least $n/6$ elements, that are separated by the cut vector. Hence, $g(2,n) \ge n/6$.

Continuing with the issue of vertex $y$ in the proof of Lemma \ref{l:lower-2}, notice that since $y$ is {\em not} a leaf, then, in the phylogenetic setting, $y$ does not belong to the subset of vertices common by $[n]$ (as only the leaves are common), hence $w(y)=0$, so the following error-term free analogue of Lemma \ref{l:lower-2} (which, in turn, also uses the aforementioned error-term free analogue of Lemma \ref{l:simple}) is:
Suppose that $T_1$ and $T_2$ are two unrooted phylogenetic trees on common taxa set $[n]$.
Then there is a cut vector $(e_1,e_2)$ and disjoint sets of taxa $X,Y$ separated by the cut vector,
where $|X| \ge 4n/27$, $|Y|\ge 4n/27$, and $|X \cup Y| \ge 4n/9$.
Finally notice that in the proof of Theorem \ref{t:lower-3}, once we use the error-term free analogue
of Lemma \ref{l:lower-2}, we obtain the error-term free analogue of Theorem \ref{t:lower-3}:
Suppose that $T_1,T_2,T_3$ are unrooted phylogenetic trees on common taxa set $[n]$.
Then there is a cut vector
$(e_1,e_2,e_3)$ and disjoint taxa sets $X,Y$ each with at least $2n/27$ elements, that are separated by the cut vector. Hence, $g(3,n) \ge 2n/27$.

As for the upper bound, we must prove that $g(2,n) \le (1+o_n(1))n/6$ and
$g(3,n) \le (1+o_n(1))2n/27$. We cannot directly use the constructions in 
the proof of Theorem \ref{t:upper} since the balanced path blowups used there contain non-leaves that are labeled by $[n]$ and, furthermore, they are not phylogenetic trees as they have (many) internal vertices of degree $2$.
Nevertheless, we can slightly modify them to obtain an analogous construction for
phylogenetic trees.

An {\em $n$-leaf caterpillar} is obtained attaching a pendant edge (and hence a pendant leaf) to each internal vertex of an $n$-vertex path. An {\em end leaf} of a caterpillar is an end vertex of a longest path of the caterpillar. Observe that caterpillars
are unrooted phylogenetic trees; in fact, they are an important class of phylogenetic trees
 --- for instance, they are used to prove the NP-hardness of the well-known maximum quartet compatibility problem \cite{steel-1992}
and they are the only general class of unrooted phylogenetic trees for which the longstanding Bandelt-Dress Conjecture has been resolved \cite{ANS-2016}.

Suppose that $H$ is a fixed unrooted phylogenetic tree on $\ell$ leaves and suppose that $n \ge \ell$. Let $\L(H)$ be the set of leaves of $H$.
Suppose that $\sum_{v \in \L(H)}x_v = n$ where each $x_v$ is a positive integer.
Let $v \in \L(H)$ and let $u$ be the neighbor of $v$ in $H$.
Replace $v$ by taking a caterpillar on $x_v+1$ leaves, identifying one of its end vertices with $u$,
See Figure \ref{f:caterpillars}.
We call the new tree a {\em caterpillar blowup} of $H$. If all $x_v$ are equal (hence equal to $n/\ell$)
we call the new tree a {\em balanced caterpillar blowup} of $H$. Notice that a caterpillar blowup has $n$ leaves and is an unrooted phylogenetic tree.

\begin{figure}[!ht]
\includegraphics[scale=0.6,trim=50 380 90 120, clip]{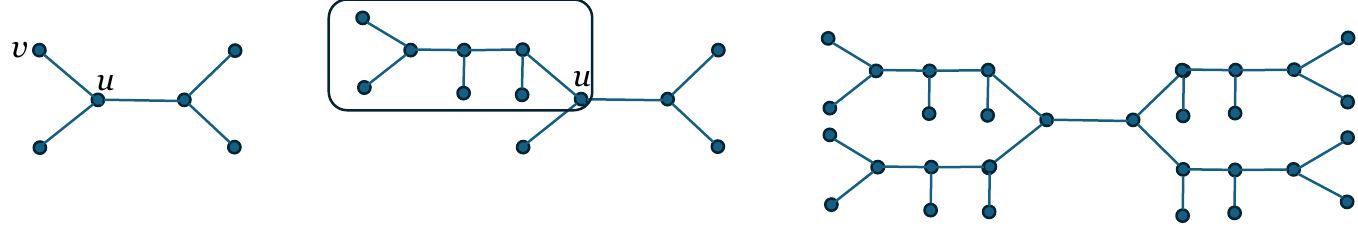}
\caption{An unrooted phylogenetic tree $H$ on the left with $\ell=4$ leaves. The leaf $v$ is replaced with a $5$-leaf caterpillar (circled) to obtain the tree in the middle; here $x_v=4$. The tree on the right is a balanced caterpillar blowup of $H$.}
\label{f:caterpillars}
\end{figure}  

Now proceed similar to the proof of Theorem \ref{t:upper}. Let $H_1$ be obtained from $K_{1,3}$ by splitting each leaf into two leaves (so $H_1$ is an unrooted phylogenetic tree with $6$ leaves).
Let $T_1$ be a balanced caterpillar blowup of $H_1$ having $n$ leaves (so $6$ divides $n$).
Let $H_2$ be the balanced double star on $4$ leaves (i.e., the quartet tree identical to the tree on the left of Figure \ref{f:caterpillars}.) Let $T_2$ be a balanced caterpillar blowup of $H_2$ having $n$ leaves (so also $4$ divides $n$, altogether $12$ divides $n$).
Randomly and bijectively assign the label set $[n]$ to the leaves of $T_1$ and, independently, also
to the leaves of $T_2$.
Exactly as in the proof of Theorem \ref{t:upper}, we have, whp, that $f(\{T_1,T_2\},n) \le n/6 +o(n)$
In turn, this implies that $g(2,n) \le (1+o_n(1))n/6$.
Similarly, take three copies $T_1,T_2,T_3$ of the tree $T_1$ in the previous construction.
Randomly and bijectively assign the label set $[n]$ to the leaves of $T_i$, independently for each $i=1,2,3$.
Exactly as in the proof of Theorem \ref{t:upper}, we have, whp, that $f(\{T_1,T_2,T_3\},n) \le 2n/27 +o(n)$. In turn, this implies that $g(3,n) \le (1+o_n(1))2n/27$.
\end{proof}

\section{Concluding remarks and open problems}\label{sec:concluding}

\begin{itemize}
\item 
As noted in the introduction, $f(r,k)$ is well-defined since $f(r,k,n)/n$ tends to a limit as $n$ goes to infinity.
\begin{proposition}\label{prop:limit}
	$\lim_{n \rightarrow \infty} f(r,k,n)/n$ exists.
\end{proposition}
\begin{proof}
	Let	$s(r,k) = \liminf_{n \rightarrow \infty} f(r,k,n)/n$.
	Let $\varepsilon > 0$. We must prove that for all $n$ sufficiently large,
	$f(r,k,n) \le (s(r,k)+\varepsilon)n$ as this will prove that
	$\limsup_{n \rightarrow \infty} f(r,k,n)/n \le s(r,k)+\varepsilon$ and hence the claimed limit exists.
	
	So, let $\varepsilon > 0$ be given. By the definition of $\liminf$, there exist infinitely many 
	positive integers $q$ where for each such $q$ there exists a family of trees $\T = \{T_1,\ldots,T_k\}$  each with a common subset of vertices labeled by $[q]$ and each with maximum degree at most $r$, such that $f(\T,q) \le (s(r,k)+\varepsilon/2)q$. Choose such a $q$ with $q \ge 4/\varepsilon$ and suppose than $n > q^2$.
	
	We next prove that the exists a family of trees $\Q = \{Q_1,\ldots,Q_k\}$ each with a common subset of vertices labeled by $[n]$ and each with maximum degree at most $r$, such that $f(\Q,n) \le (s(r,k)+\varepsilon)n$; in particular, $f(r,k,n) \le (s(r,k)+\varepsilon)n$, as required.  Construct $Q_i$ from $T_i$ as follows.
	Let $d=\lfloor n/q \rfloor$.
	Consider the rooting of $T_i$ at the vertex whose label is $q$, so now every vertex but the root has a unique parent.
	For each $t \in [q-1]$ consider the vertex of $T_i$ whose label is $t$, and the edge connecting it with its parent. Subdivide this edge $d-1$ times (this introduces $d-1$ vertices of degree $2$), assign the labels $td,td-1,\ldots,td-d+2$ to these new vertices, and relabel $t$ with the label $td-d+1$.
	This now defines a tree with the labels $1,\ldots (q-1)d$. As for the root (whose original label is
	$q$), relabel it to $qd-d+1$, attach a new path of length $n-(qd-d+1)$ to the root (this introduces
	$n-(qd-d+1)-1$ new vertices of degree $2$ and a new leaf), and label them with the labels
	$qd-d+2,\ldots,n$. Observe also that the number of vertices on this new path is
	$n-(qd-d+1)=n-1-d(q-1) \le n-1-(n/q-1)(q-1) \le d+q \le 2d$.
	Overall, we obtain the tree $Q_i$ which has a common subset of vertices labeled by $[n]$ and has maximum degree at most $r$. Also notice that each edge of $Q_i$ is either an original edge of $T_i$, or is an edge obtained by subdividing an edge $T_i$. As each such edge of $T_i$ is subdivided at most $d-1$
	times (except for the path attached to the root whose length is at most $2d$) and are consecutively labeled, this immediately implies that $f(\Q,n) \le d f(\T,q)+2d$. We therefore have
	$$
	f(\Q,n) \le d f(\T,q)+2d \le d(s(r,k)+\varepsilon/2)q + 2d \le n(s(r,k)+\varepsilon/2)+2d
	\le n(s(r,k)+\varepsilon)
	$$
	where in the last inequality we have used that $\varepsilon n/4 \ge d$ since $q \le 4/\varepsilon$.
\end{proof}

\item
Theorem \ref{t:upper} proves, in particular, that $f(r,3) \le (r-1)/r^3$. For $r=3$ (and $r=2$, see below) we know that this upper bound is tight. Thus, Conjecture \ref{conj:1} asserts that this bound is tight for all $r$. Notice that in order to prove Conjecture \ref{conj:1}, one might need to generalize the already rather technical Lemma \ref{l:lower-2}.

We can generalize our upper bound constructions from Theorem \ref{t:upper} to an arbitrary number of $k$ trees as follows. If $k$ is odd, we may take $k$ copies of $T_1$ (recall that when $k=3$ we have taken three copies of $T_1$), randomly and bijectively assigning labels from $[n]$ to the $n$ non-backbone vertices in each tree, independently. Using Lemma \ref{l:random} this construction yields that whp $f(\{T_1,\ldots,T_k\},n) \le
n(r-1)^{(k-1)/2}/r^k+o(n)$. Similarly, if $k$ is even, we may take $k-1$ copies of $T_1$ and one copy of $T_2$
(recall that in Theorem \ref{t:upper} we have taken $T_2$ as the balanced path blowup of $H_2$)
and randomly and independently assign labels to the non-backbone vertices in each tree. Using Lemma \ref{l:random} this construction yields that whp $f(\{T_1,\ldots,T_k\},n) \le
n(r-1)^{(k-2)/2}/2r^{k-1}+o(n)$. However, we have no idea if the respective values
($(r-1)^{(k-1)/2}/r^k$ for odd $k$ and $(r-1)^{(k-2)/2}/2r^{k-1}$ for even $k$) equal $f(r,k)$,
so we leave it as an open problem.

\item
As mentioned in the introduction, proving $f(2,k)=1/2^k$ is rather easy. For completeness, we provide the proof in the following proposition.
\begin{proposition}\label{prop:2k}
	$f(2,k)=1/2^k$.	
\end{proposition}
\begin{proof}
	The lower bound is quite similar to the proof of Lemma 2.1 of \cite{AM-1986}.
	Suppose that $T_1, \ldots, T_k$ are paths on vertex set $[n]$. We may equivalently view $T_i$ as a permutation $\pi_i$ of $[n]$, so that
	$\pi_i(j)$ is the label of vertex $j$ on the path $T_i$.
	For $1 \le i \le k$, we construct a cut vector $(e_1,\ldots,e_i)$ 
	and disjoint sets of vertices $A_i$ and $B_i$ with $|A_i|, |B_i| \ge n(1-o_n(1))/2^i$ such that $A_i$ and $B_i$ are separated by $(e_1,\ldots,e_i)$.
	We proceed inductively. Let $e_1$ be the first edge on $T_1$ for which the size
	of the set of vertices $A_1$ appearing in $\pi_1$ before $e_1$ is $\lceil n/2 \rceil$. The remaining set of vertices (appearing after $e_1$ in $\pi_1$) is $B_1$, so we have
	$|B_1|=\lfloor n/2 \rfloor$, and $A_1,B_1$ are separated by $e_1$.
	Assume that we have already defined $A_i$, $B_i$
	having the claimed properties, in particular $|A_i|,|B_i| \ge n(1-o_n(1))/2^i$.
	Let $\sigma$ denote the restriction of $\pi_{i+1}$ to $A_i \cup B_i$,
	so $\sigma$ is a permutation of $A_i \cup B_i$ and $|A_i \cup B_i| \ge n(1-o_n(1))/2^{i-1}$.
	Let $C$ denote the smallest prefix of $\sigma$ for which $|C| \ge n/2^i$ and let $D$ be the remaining vertices of $\sigma$. Now, suppose first that $|A_i \cap C| \ge n/2^{i+1}$. Then we must have that $|B_i \cap D| \ge n/2^{i+1}-o(n)$
	so we may set $A_{i+1}=A_i \cap C$ and $B_{i+1}=B_i \cap D$.
	Otherwise, we must have $|A_i \cap D| \ge n/2^{i+1}-o(n)$. Then we must have that $|B_i \cap C| \ge n/2^{i+1}-o(n)$ so we may set $A_{i+1}=A_i \cap D$ and $B_{i+1}=B_i \cap C$.
	Finally, let $e_{i+1}$ be the first edge on $T_{i+1}$ which separates $A_{i+1}$ and $B_{i+1}$.
	Once we arrive at $i=k$, we have that $A_k$ and $B_k$ are separated by $(e_1,\ldots,e_k)$
	and each has size at least $n(1-o_n(1))/2^k$, proving that $f(2,k) \ge 1/2^k$.
	
	For the upper bound, we can just take $k$ random permutations of $[n]$ (and the corresponding $n$-vertex paths $T_1,\ldots,T_k$ they define) and use Lemma \ref{l:random} to easily obtain that whp
	$f(\{T_1,\ldots,T_k\},n) \le n/2^k+o(n)$, whence $f(2,k) \le 1/2^k$.
\end{proof}

\item
Another problem of interest in the phylogenetic tree setting is the following ''inverse''
question. Determine the smallest possible $n$, such that for any set of $k$ unrooted phylogenetic trees on
taxa set $[n]$, there is a quartet common to all of them. In other words, determine the smallest $n$
for which $g(k,n) \ge 2$. Denote this number by $h(k)$ and let $h^*(k)$ be analogously defined when
restricting the trees to be $n$-leaf caterpillars (which, recall, are an important family of unrooted phylogenetic trees). It is proved in \cite{BD-1986} (see also \cite{ASY-2014} Lemma 5.2), that $h(2)=h^*(2)=6$ (notice that Theorem \ref{t:taxa} only gives $h(2) \le 7$). It may be interesting to determine $h(k)$ in general.
It is not too difficult to prove that $h^*(k) = 2^k+2$. Indeed, for the lower bound, just mimic the lower bound proof of Proposition \ref{prop:2k} and observe that there is
no need for the $o_n(1)$ term in this case; hence any set of $k$ permutations of order $2^k+2$ separates two disjoint sets of two elements each, which corresponds to a common quartet in a set of $k$ caterpillars.
For the upper bound, constructing a family of $k$ caterpillars on $2^k+1$ leaves having no common quartet can be done recursively; we omit the details, but for example, if $k=3$ just consider the
three permutations (correspondingly, caterpillars) $1 2 3 4 5 6 7 9$, $1 2 7 6 5 4 3 8 9$, $1 8 3 6 5 4 7 2 9$. Determining $h(k)$ though, seems considerably more involved. However, while $h(2)=h^*(2)=6$, we already have $h(3) > h^*(3)=10$, as shown in Figure \ref{f:h3}.
\end{itemize}

\begin{figure}[!ht]
	\includegraphics[scale=0.6,trim=50 320 90 110, clip]{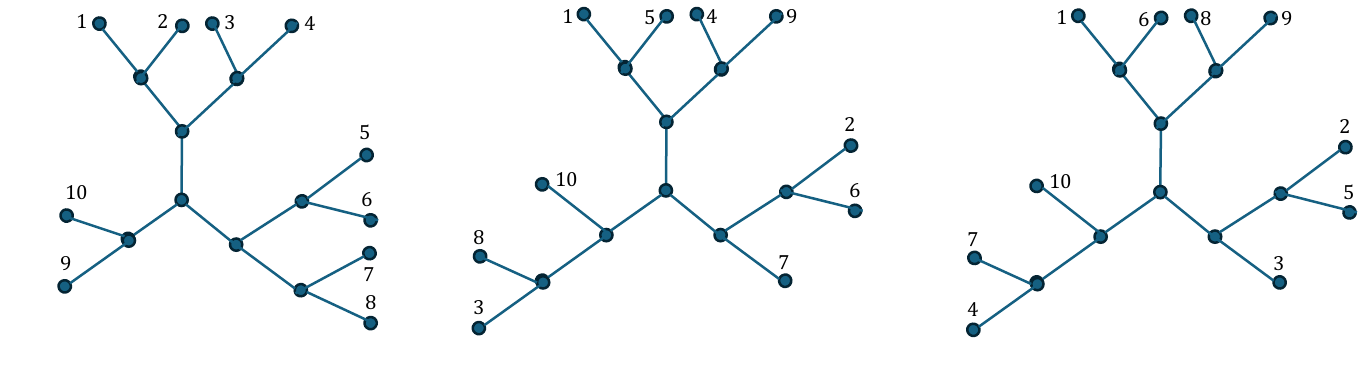}
	\caption{Three unrooted phylogenetic trees, each on taxa set $[10]$ with no common quartet (i.e., no
		disjoint taxa sets $X,Y$ of size $2$ each that are separated in all three trees).}
	\label{f:h3}
\end{figure}  

\bibliographystyle{abbrv} 
\bibliography{references} 


%
%

\end{document}